\documentclass[preprint,3p,times]{elsarticle}

\usepackage{amssymb}
\usepackage{amsmath}
\usepackage{amsthm}
\usepackage{color}
\usepackage{graphicx,subfig} 
\usepackage{booktabs}   
\usepackage{threeparttable} 
\usepackage{enumerate}
\usepackage[all,pdf]{xy}
\usepackage{lmodern,amsmath}
\usepackage{amsthm}
\usepackage{tikz-cd}
\usepackage{mathrsfs}
\usepackage{lipsum}
\usepackage{enumitem,verbatim}
\usepackage{tikz}
\usepackage{appendix}
\usepackage[colorlinks=true]{hyperref}

\theoremstyle{plain}
  \newtheorem{thm}{Theorem}[section]
  \newtheorem{lem}[thm]{Lemma}
  \newtheorem{prop}[thm]{Proposition}
  \newtheorem{cor}[thm]{Corollary}
\theoremstyle{definition}
  \newtheorem{definition}[thm]{Definition}
  \newtheorem{exmp}[thm]{Example}
  \newtheorem{remark}[thm]{Remark}

\DeclareMathAlphabet{\mathcal}{OMS}{cmsy}{m}{n}

\makeatletter
\def\ps@pprintTitle{%
 \let\@oddhead\@empty
 \let\@evenhead\@empty
 \def\@oddfoot{\centerline{\thepage}}%
 \let\@evenfoot\@oddfoot}
\makeatother

\renewcommand{\phi}{\varphi}

\newcommand{\Luk}{\mathrm{L}}

\numberwithin{equation}{section}

\allowdisplaybreaks

\begin{document}

\begin{frontmatter}



\title{Completeness of the Fuzzy Order on Fuzzy Numbers
}



\author{Mingchun Xie\corref{cor}}
\ead{xiemc1011@163.com}

\cortext[cor]{Corresponding author.}
\address{School of Mathematics, Sichuan University, Chengdu 610064, China}

\begin{abstract}
This paper investigates the completeness properties  of the fuzzy order on the fuzzy number space $\mathbb{E}^1$ within the framework of $[0,1]$-enriched categories. \(S\)-convergence under the three fundamental continuous t-norms is characterized, and inclusion relations among the induced open ball topologies are established. It is shown that $(\mathbb{E}^1,P)$ is Cauchy complete if and only if the t-norm is not isomorphic to the {\L}ukasiewicz t-norm, whereas $(\mathbb{E}^1,P)$ fails to be Smyth complete for every continuous t-norm. Consequently, the symmetrization $(\mathbb{E}^1,S)$ is Smyth complete if and only if the underlying continuous t-norm is not isomorphic to the {\L}ukasiewicz t-norm.

\end{abstract}

\begin{keyword}
Fuzzy numbers\sep [0,1]-enriched category\sep Cauchy complete\sep Smyth complete\sep Continuous t-norm  

\end{keyword}

\end{frontmatter}

\section{Introduction}
Fuzzy set theory, introduced by Zadeh \cite{ZADEH1965338}, provides a fundamental mathematical framework for modeling vagueness and uncertainty. Among its central objects of study, fuzzy numbers generalize real numbers and play a role in fuzzy mathematics analogous to that of real numbers in classical analysis.   They offer a natural model for imprecise quantitative information. The problem of comparing and ordering fuzzy numbers is therefore foundational: it directly affects the reliability of applications such as fuzzy optimization and multi-attribute decision-making.
 
 Existing approaches to comparing fuzzy numbers, including defuzzification \cite{BORTOLAN19851,CHU2002111,KaufmannGupta1985}, reference-set construction \cite{CHEN1985113,jain1976decision,jain1977procedure}, and fuzzy relation approaches \cite{wang1997comparative,wang1996classification}, face a fundamental trade-off: they either sacrifice the strict mathematical structure of an order to achieve a total order, or they employ fuzzy relations that lack the support of complete order axioms. What is still missing is a fuzzy analogue of the classical partial order: a construct that satisfies the fuzzy analogues of reflexivity and transitivity and that is rich enough to support a substantial theory. Although the abstract notion of a fuzzy order has been defined, the construction of a concrete and natural instance for fuzzy numbers remained an open problem until the recent work of Kou and Xie \cite{kouhui}.

Recently, Kou and Xie introduced a fuzzy order \(P\) on the space of fuzzy numbers, based on the theory of \([0,1]\)-enriched categories. This relation naturally extends both the usual order on the real line and the usual order on interval numbers. Moreover, it allows the degree of comparison to be modulated by the choice of a continuous t-norm and reflects both the distribution and the structural information of fuzzy numbers. They proved that \((\mathbb{E}^1_{[-M,M]},P)\) is Yoneda complete for uniformly bounded fuzzy numbers and is complete if and only if the continuous t-norm is the G\"odel t-norm, while leaving the various completeness properties of the full space \((\mathbb{E}^1,P)\) open.

The aim of this paper is to resolve these open questions. The main contributions are as follows.
\begin{enumerate}[label={(\arabic*)}]
\item We show that, under the {\L}ukasiewicz t-norm, \(S\)-convergence coincides with convergence in the supremum metric
\(
d_{\infty}(u,v)=\|u-v\|_{\infty},
\)
and we characterize \(S\)-convergence under the product and G\"odel (minimum) t-norms; consequently, the corresponding open ball topologies satisfy
\(
\tau_{\Luk}\subseteq\tau_{*}\subseteq\tau_{\wedge}.
\)
 \item We establish a complete characterization of the completeness properties of the fuzzy order. For every continuous \(t\)-norm, \((\mathbb{E}^1,P)\) is not Smyth complete, while it is Cauchy complete if and only if the \(t\)-norm is not isomorphic to the {\L}ukasiewicz \(t\)-norm. Consequently, the symmetrization \((\mathbb{E}^1,S)\) is Smyth complete if and only if the underlying \(t\)-norm is not isomorphic to the {\L}ukasiewicz \(t\)-norm. Thus, the {\L}ukasiewicz \(t\)-norm is the unique obstruction to Cauchy completeness of \((\mathbb{E}^1,P)\) and Smyth completeness of \((\mathbb{E}^1,S)\), whereas the failure of Smyth completeness of \((\mathbb{E}^1,P)\) is universal.

\end{enumerate}

The rest of this paper is organized as follows. Section~2 recalls the necessary preliminaries on fuzzy numbers, continuous t-norms, and \([0,1]\)-enriched categories. Section~3 characterizes \(S\)-convergence under the three fundamental continuous t-norms and establishes the inclusion relations among the induced open ball topologies. Section~4 investigates the completeness of the fuzzy order. Section~5 concludes.

\section{Preliminaries} 
  This section collects the basic notions required throughout the paper: first,
 the definition and elementary properties of fuzzy numbers; then, triangular norms(t-norms) and their residuated implications—the logical operators underlying the fuzzy order on \(\mathbb{E}^1\) recently introduced by Kou and Xie; finally, basic concepts from the theory of \([0,1]\)-enriched categories.
\begin{definition}
The fuzzy number space \( \mathbb{E}^1 \) is the set of all functions \( u :\mathbb{R}\to [0,1]\) satisfying the following properties:
\begin{enumerate}[label={\textup{(\arabic*)}}]
  \item Normality: There exists \( x_0 \in \mathbb{R} \) with \( u(x_0) = 1 \).
  \item Convexity: \( u(\lambda x + (1-\lambda)y) \geq \min \{u(x), u(y)\} \) for all \( x, y \in \mathbb{R}, \lambda \in [0, 1] \).
  \item Upper-semicontinuity: \( u(x) \) is upper-semicontinuous, i.e., for all  \(x_0\in \mathbb{R}\), \(
\limsup_{x \to x_0} u(x) \le u(x_0).\) 
  \item Compact support: \([u]_0=\operatorname{cl}_{\mathbb{R}} \left\{x \in \mathbb{R} : u(x) > 0\right\} \) is compact in \( \mathbb{R} \).
\end{enumerate}
For \(M>0\), let
\[
\mathbb{E}^1_{[-M,M]}
 :=\{u\in\mathbb{E}^1:[u]_0\subseteq[-M,M]\}.
\]
\end{definition} 
Given a fuzzy subset \( u \) on \( \mathbb{R} \), the \( \alpha \)-level set of \( u \) is defined by \([u]_ \alpha  = \{x \in \mathbb{R} : u(x) \geq  \alpha \} \) for \(  \alpha  \in (0, 1] \). 
\begin{definition}
    Let $u\in \mathbb{E}^1$, and define its left part \( u^\ell \) and right part \( u^r \) as functions  \( \mathbb{R} \to [0,1] \) by:
\[
u^\ell(x) =
\begin{cases}
	0, & x < u_0^-; \\
	u(x), & x \in [u_0^-, u_1^-]; \\
	1, & x > u_1^-.
\end{cases}
\qquad
u^r(x) =
\begin{cases}
	1, & x < u_1^-; \\
	u(x), & x \in [u_1^-, u_0^+]; \\
	0, & x > u_0^+.
\end{cases}
\]
Here \(u^-_0\) and \(u^+_0\) denote the left and right endpoints of the \(0\)-level set, so that \([u]_0=[u^-_0,u^+_0]\); likewise, \(u^-_1\) and \(u^+_1\) denote the left and right endpoints of the \(1\)-level set \([u]_1=[u^-_1,u^+_1]\).
\end{definition}
\begin{prop}[\cite{kouhui}, Proposition 3.2]\label{represent of fuzzy number}
    Let \( u \in \mathbb{E}^1 \) and \( u = (u^\ell, u^r) \). Then:
\begin{enumerate}[label={\textup{(\arabic*)}}]
    \item \( u^\ell,  u^r\)  are both upper-semicontinuous functions. \( u^\ell\) is a non-decreasing function on \( \mathbb{R} \) and  \(u^r\) is a non-increasing function on \( \mathbb{R} \);
    \item There exist \(x_0,x_1\in \mathbb{R}\) such that $u^\ell(x_0)=u^r(x_1)=0$;
    \item  For every \( x\in \mathbb{R} \), \( u^\ell(x)\vee u^r(x)=1\).
\end{enumerate}
Conversely, if a pair of functions $(f,g)$ satisfies \((1)-(3)\), then there is a unique $u\in\mathbb{E}^1$ such that $u=f\wedge g$ pointwise, and then $u^\ell=f$, $u^r=g$.

\end{prop}

\begin{definition}\cite{Alsina2006}
    A \( t \)-norm is a function of two variables \( \& : [0,1] \times [0,1] \to [0,1] \) that satisfies the following properties:
\begin{enumerate}[label={\textup{(\arabic*)}}]
    \item Commutativity: for every \( x, y \in [0,1] \) one has \( x \&y = y\&x \),
    \item Monotonicity: for every \( x, y, z, w \in [0,1] \) if \( x \leq z \), \( y \leq w \) then  $x \& y \leq z\& w$,
    \item Associativity: for every \( x, y, z \in [0,1] \) one has \( x\& (y\& z) = (x\& y)\& z \),
    \item Identity element: for every \( x \in [0,1] \) one has \( x\& 1 = x \).
\end{enumerate}
\end{definition}

Given a left-continuous t-norm \(\&\), the symbol \(\rightarrow\) denotes its residuated implication,
\[
x\rightarrow y = \sup\{u \in [0,1] \mid x\mathbin{\&} u \leq y\},
\]
and \(\leftrightarrow\) denotes the corresponding biresiduum, defined by
\[
x\leftrightarrow y=(x\rightarrow y)\wedge (y\rightarrow x).
\]
We briefly review the  important properties of the residuated implication and bi-implication, which can be found in \cite{HL2006}.
\begin{prop}\label{prop:residuum-properties}
    Suppose that \( \& \) is a left-continuous t-norm and \( \rightarrow \) and \( \leftrightarrow \) denote the residuated implication and bi-implication with respect to \( \& \). Then
\begin{enumerate}[label={\textup{(\arabic*)}}]
    \item \( x \leq y \) if and only if \( x\rightarrow y = 1 \).
    \item \( x\& y \leq z \) if and only if \( x \leq y\rightarrow z \).
    \item \( (x\rightarrow y)\& (y\rightarrow z) \leq x\rightarrow z \).
    \item \( (x\leftrightarrow y)\& (y\leftrightarrow z) \leq x\leftrightarrow z \).
    \item \( 1\rightarrow y = y \).
    \item \( x\&( x \rightarrow y) \leq y \).
    \item \(( \bigvee_{i} x_i)\rightarrow y=\bigwedge_i(x_i\rightarrow y)\).
    \item \( x\rightarrow(\bigwedge_i y_i) =\bigwedge_i (x\rightarrow y_i)\).
\end{enumerate}
\end{prop}
Kou and Xie introduced a natural fuzzy order on \(\mathbb{E}^1\)
based on \([0,1]\)-enriched category theory. We recall the basic
notions and terminology needed below; see
\cite{Zhang2024IntroductoryNO} for further details.
\begin{definition}
	A \([0,1]\)-enriched category is a pair $(X, \alpha)$, where $X$ is a set and $\alpha: X \times X \to [0,1]$ is a function such that
	\begin{enumerate}[label={\textup{(\arabic*)}}]
		\item $\alpha(x, x) = 1$ for all $x \in X$;
		\item  $\alpha(y, z) \& \alpha(x, y) \leq \alpha(x, z)$ for all $x, y, z \in X$.
	\end{enumerate}
\end{definition}
For notational convenience, when dealing with a \([0,1]\)-enriched category \((X,\alpha)\), the symbol \(\alpha\) is often omitted and we write $X(x,y)$ for $\alpha(x,y)$.
Two elements $x$ and $y$ of a \([0,1]\)-enriched category $X$ are isomorphic if $X(x,y) = 1 = X(y,x)$. A \([0,1]\)-enriched category $X$ is separated if its isomorphic elements are identical.

Let $X$ be a \([0,1]\)-enriched category. For each $x \in X$ and $r < 1$, the set
\[
B(x, r) := \{y \in X \mid X(x, y) > r\}
\]
is called the open ball of $X$ with center $x$ and radius $r$. The collection
$\{B(x, r) \mid x \in X, r < 1\}$ is a base for a topology on $X$, the resulting topology \cite{J.Jacas1995}  is called the open ball topology of $X$.

For each \([0,1]\)-enriched category $(X, \alpha)$, the symmetrization of $(X, \alpha)$ refers to the symmetric  \([0,1]\)-enriched category $(X, S(\alpha))$, where
\[
S(\alpha)(x, y) = \min\{\alpha(x, y), \alpha(y, x)\}.
\]
It is clear that the open ball topology of the symmetrization of $X$ is the least common refinement of the open ball topology of $X$ and that of $X^{\text{op}}$.

\begin{prop}[\cite{Zhang2024IntroductoryNO}, Proposition 8.9]
     A net  \(\{x_i\}_{i\in D}\) of a \([0,1]\)-enriched category \(X\) converges to \(x\)  with respect to  the open ball topology if and only if \[ \sup_{i\in D}\inf_{j\ge i} X(x,x_j)=1.\]
\end{prop}

\begin{definition}[\cite{kouhui}, Definition 3.3]
Let \( u, v \in \mathbb{E}^1 \) and let\( \rightarrow \) be the implication of a  left-continuous t-norm \( \& \). Define
\[
P(u, v) = \inf_{x \in \mathbb{R}} \left[ (v^\ell(x) \rightarrow u^\ell(x)) \wedge (u^r(x) \rightarrow v^r(x))\right],\quad
S(u,v)=P(u,v)\wedge P(v,u).
\]
 \((\mathbb{E}^1,P)\) is a \([0,1]\)-enriched category, i.e.\ \(P(u,u)=1\) and \(P(v,w)\& P(u,v)\le P(u,w)\) for all \(u,v,w\in\mathbb{E}^1\).
Hence, \((\mathbb{E}^1,S)\) is the symmetrization of \((\mathbb{E}^1,P)\). 

\end{definition}
\begin{prop}
	 Let $\&$ be a left-continuous t-norm. \((\mathbb{E}^1, S)\) is  a  separated \([0,1]\)-enriched category.
\end{prop}
\begin{proof}
Reflexivity is  immediate. For transitivity, take \( u, v, w \in \mathbb{E}^1 \). 
Using the properties of the implication, one verifies as follows:
\begin{align*}
    S(u, w) \& S(w, v) &=\inf_{x\in \mathbb{R}} \left[(w^\ell(x) \leftrightarrow u^\ell(x)) \wedge (u^r(x) \leftrightarrow w^r(x)) \right] \& \inf_{x\in \mathbb{R}} \left[ (v^\ell(x) \leftrightarrow w^\ell(x)) \wedge (w^r(x) \leftrightarrow v^r(x))\right] \\
    &\le \inf_{x\in \mathbb{R}} \left\{\left[(w^\ell(x) \leftrightarrow u^\ell(x)) \wedge (u^r(x) \leftrightarrow w^r(x)) \right]\& \left[(v^\ell(x) \leftrightarrow w^\ell(x)) \wedge (w^r(x) \leftrightarrow v^r(x))\right]\right\} \\
    &\le \inf_{x\in \mathbb{R}} (v^\ell(x) \leftrightarrow u^\ell(x)) \wedge (u^r(x) \leftrightarrow v^r(x))=S(u,v).
\end{align*}
It is straightforward to verify that \(S(u,v)=1\) implies \(u=v\). 
\end{proof}

\section{Convergence for the three fundamental continuous t-norms}
It is well known that every continuous t-norm admits a unique representation
as an ordinal sum of continuous Archimedean t-norms, each of which is isomorphic
either to the {\L}ukasiewicz t-norm or to the product t-norm; the G\"odel
(minimum) t-norm corresponds to the degenerate case in which every element is
idempotent~\cite[Theorem~2.4.3]{Alsina2006}. In this section we restrict attention to the three fundamental continuous t-norms.
For each of them we characterize \(S_{\&}\)-convergence on \(\mathbb{E}^1\), that is, convergence in the open
ball topology of the symmetrization. As a consequence, we establish a chain of inclusions
among the induced open ball topologies,
\(\tau_{\Luk}\subseteq \tau_{*}\subseteq \tau_{\wedge}\).
The three fundamental continuous t-norms and their residua are given as follows:
\begin{align*}
& a \& b = \max(a + b - 1, 0); \\
& a \to c =  \min(1,\, 1-a+c).
&& (\text{\(\&=\Luk\), {\L}ukasiewicz}) \\[1em]
& a \& b = \min(a, b); \\
& a \to c =\begin{cases}
    c, &a>c;\\
    1, &otherwise.
\end{cases} 
&& (\text{\(\&=\wedge\), G\"odel}) \\[1em]
& a \& b = a \cdot b; \\
& a \to c = \begin{cases}
    c / a, &a>c;\\
    1,&otherwise.
\end{cases}
&& (\text{\(\&=*\), Product})
\end{align*}

The unit interval carries a $[0,1]$-enriched category structure $([0,1],\alpha_L)$ with $\alpha_L(x,y)=x\to y$.  The symmetrization of \(([0,1], \alpha_L)\) is  denoted by \(([0,1], S)\), where \(S(x,y)=x\leftrightarrow y\). The open ball topology induced by the bi-implication on the unit interval has already been investigated in the literature \cite{M. Krupka2025, Zhang2010}. 
\begin{remark}
    The following are bases of open ball topologies of the three fundamental continuous t-norms.
\begin{align*}
 \sigma_1 &= \{[0,b) \mid b>0\} \cup \{(a,b) \mid a<b\} \cup \{(a,1] \mid a<1\} \quad \text{for the {\L}ukasiewicz $t$-norm}, \\
\sigma_2 &= \sigma_1 \cup \{\{0\}\} \quad \text{for the product $t$-norm}, \\
\sigma_3 &= \{\{a\} \mid a<1\} \cup \{(a,1] \mid a<1\} \quad \text{for the G\"odel $t$-norm}.
\end{align*}
\end{remark}
\begin{definition}
   Let $\&$ be a continuous t-norm. A net $\{u_i\}_{i\in D}$ in $\mathbb{E}^1$ is called $S$-convergent to $u$ if it converges to $u$ with respect to the open ball topology of the symmetrization $S$; that is, if
\[
\sup_{i\in D}\inf_{j\ge i} S(u,u_j)=1.
\]

The open ball topology induced by a continuous t-norm \(\&\) is denoted by \(\tau_{\&}\).
\end{definition}
\begin{definition}
\label{def:convergence}
Let $X$ be a nonempty set and $(Y,d)$ a metric space. 
Let $\{f_i\}_{i\in D}$ be a net of functions $f_i:X\to Y$ and let $f:X\to Y$.
\begin{enumerate}[label=(\roman*)]
    \item \textbf{Pointwise convergence.} 
    The net $\{f_i\}_{i\in D}$ is said to \emph{converge pointwise} to $f$ on $X$ if
    \[
    \forall\,x\in X,\;\forall\,\varepsilon>0,\;\exists\,i_0\in D,
    \;\forall\,i\geq i_0:\quad d\bigl(f_i(x),\,f(x)\bigr)<\varepsilon.
    \]
    \item \textbf{Uniform convergence.}
    The net $\{f_i\}_{i\in D}$ is said to \emph{converge uniformly} to $f$ on $X$ if
    \[
    \forall\,\varepsilon>0,\;\exists\,i_0\in D,\;
    \forall\,x\in X,\;\forall\,i\geq i_0:\quad d\bigl(f_i(x),\,f(x)\bigr)<\varepsilon.
    \]
\end{enumerate}
\end{definition}
\begin{remark}\label{rmk:first-countable}
In the special case $Y=[0,1]$ endowed with the usual metric $d(a,b)=|a-b|$, 
uniform convergence of $\{f_i\}_{i\in D}$ to $f$ is equivalent to 
$\|f_i-f\|_{\infty}\to 0$, where $\|\cdot\|_{\infty}$ denotes the supremum norm
$\|g\|_{\infty}=\sup_{x\in X}|g(x)|$.
\end{remark}

\begin{lem}\label{pointwise representation}
Let $\&$ be a left-continuous t-norm  and let $u,v\in\mathbb{E}^1$. Then
\[
S(u,v)=\inf_{x\in\mathbb{R}}\bigl(u(x)\leftrightarrow v(x)\bigr).
\]
\end{lem}

\begin{proof}
We first prove the inequality
\[
\inf_{x\in\mathbb{R}}\bigl(u(x)\leftrightarrow v(x)\bigr)\ge S(u,v).
\]
From the definition of $S$ and the identity
$\inf_x(A_x\wedge B_x)=\inf_x A_x\wedge\inf_x B_x$, we have
\[
S(u,v)=\inf_{x\in\mathbb{R}}
\bigl[(u^\ell(x)\leftrightarrow v^\ell(x))\wedge
(u^r(x)\leftrightarrow v^r(x))\bigr].
\]
Fix $x\in\mathbb{R}$ and put
\[
a=u^\ell(x),\quad b=u^r(x),\quad c=v^\ell(x),\quad d=v^r(x),
\]
so that $u(x)=a\wedge b$ and $v(x)=c\wedge d$. Let
$r=(a\leftrightarrow c)\wedge(b\leftrightarrow d)$. Since $r\le a\to c$
and $r\le b\to d$, Proposition
\ref{prop:residuum-properties}(2) gives
$a\mathbin{\&}r\le c$ and $b\mathbin{\&}r\le d$. As $a\wedge b\le a$
and $a\wedge b\le b$, the monotonicity of $\&$ yields
\[
(a\wedge b)\mathbin{\&}r\le a\mathbin{\&}r\le c
\quad\text{and}\quad
(a\wedge b)\mathbin{\&}r\le b\mathbin{\&}r\le d,
\]
hence $(a\wedge b)\mathbin{\&}r\le c\wedge d$. Applying Proposition
\ref{prop:residuum-properties}(2) again,
we obtain $r\le(a\wedge b)\to(c\wedge d)$. Interchanging the roles of
$(a,b)$ and $(c,d)$ gives the symmetric inequality
$r\le(c\wedge d)\to(a\wedge b)$. Therefore
\[
(a\wedge b)\leftrightarrow(c\wedge d)
\ \ge\ (a\leftrightarrow c)\wedge(b\leftrightarrow d)=r,
\]
and taking the infimum over $x\in\mathbb{R}$ yields the desired
inequality.

It remains to prove the reverse inequality. Set
\[
\delta(s,t):=1-(s\leftrightarrow t)\qquad (s,t\in[0,1]),
\]
and write
\[
\delta_L(x):=\delta\bigl(u^\ell(x),v^\ell(x)\bigr),\qquad
\delta_R(x):=\delta\bigl(u^r(x),v^r(x)\bigr),\qquad
D:=\sup_{x\in\mathbb{R}}\max\bigl\{\delta_L(x),\delta_R(x)\bigr\}.
\]
Since $S(u,v)=1-D$, it suffices to prove
\begin{equation}\label{eq:reverse-delta}
\sup_{x\in\mathbb{R}}\delta(u(x),v(x))\ge D.
\end{equation}

\noindent\textbf{Case 1} ($D=0$). Since $\delta\ge0$ pointwise,
\eqref{eq:reverse-delta} is immediate.

\medskip
\noindent\textbf{Case 2} ($D>0$). Fix $0<\varepsilon<D$. By the
definition of $D$ as a supremum, there exists $x_0\in\mathbb{R}$ such
that
\[
\max\bigl\{\delta_L(x_0),\delta_R(x_0)\bigr\}>D-\varepsilon.
\]
We treat the two possibilities separately.

\medskip
\noindent\textbf{Case 2a} ($\delta_L(x_0)>D-\varepsilon$). Put
\[
p=u^\ell(x_0),\qquad q=v^\ell(x_0).
\]
Since $\delta(p,q)=\delta_L(x_0)>D-\varepsilon>0$, we have $p\ne q$.
As $\delta$ is symmetric and the inequality \eqref{eq:reverse-delta} is
symmetric in $u$ and $v$, we may assume $p<q$; hence $p<1$. From the
normality condition $u^\ell(x_0)\vee u^r(x_0)=1$ we get $u^r(x_0)=1$,
and therefore $u(x_0)=p$.

\noindent\textbf{Case 2a(i)} ($q<1$). Then
$v^\ell(x_0)\vee v^r(x_0)=1$ forces $v^r(x_0)=1$, so $v(x_0)=q$.
Consequently
\[
\delta(u(x_0),v(x_0))=\delta(p,q)>D-\varepsilon.
\]

\noindent\textbf{Case 2a(ii)} ($q=1$). Since $s\leftrightarrow1=s$, we
have $\delta(p,1)=1-p$.

\noindent\textbf{Case 2a(ii)-1} ($v^r(x_0)=1$). Then $v(x_0)=1$, and
therefore
\[
\delta(u(x_0),v(x_0))=\delta(p,1)=1-p>D-\varepsilon.
\]

\noindent\textbf{Case 2a(ii)-2} ($v^r(x_0)<1$). Let $y=v^-_1$ be the
left endpoint of the core $[v]_1=\{x:v(x)=1\}$. Since $v^\ell$ is
non-decreasing and $v^\ell(x_0)=q=1$, we have $y\le x_0$, and
$v(y)=1$ because $y$ belongs to the core. As $u^\ell$ is
non-decreasing, $u^\ell(y)\le u^\ell(x_0)=p<1$; normality then gives
$u^r(y)=1$, and hence $u(y)=u^\ell(y)$. Therefore
\[
\delta(u(y),v(y))
=1-u^\ell(y)
\ \ge\ 1-p
=\delta(p,1)>D-\varepsilon.
\]

\medskip
\noindent\textbf{Case 2b} ($\delta_R(x_0)>D-\varepsilon$). This case is
symmetric to Case 2a: replace the left parts by the right parts, use
that $u^r$ is non-increasing in place of $u^\ell$ being non-decreasing,
and replace the left core endpoint $v^-_1$ by the right core endpoint
$v^+_1$; the argument of Case 2a then applies verbatim.

In every case we have proved some point $z\in\mathbb{R}$ (namely
$z=x_0$ or $z=y$) with $\delta(u(z),v(z))>D-\varepsilon$, so
\[
\sup_{x\in\mathbb{R}}\delta(u(x),v(x))>D-\varepsilon.
\]
Letting $\varepsilon\to0$ proves \eqref{eq:reverse-delta}. Therefore
\[
\inf_{x\in\mathbb{R}}\bigl(u(x)\leftrightarrow v(x)\bigr)
=1-\sup_{x\in\mathbb{R}}\delta(u(x),v(x))
\le 1-D=S(u,v),
\]
which completes the proof.
\end{proof}

\begin{prop}\label{Lukasiewicz t-norm in E^1}
    Let \(\&\) be the {\L}ukasiewicz t-norm. \(S_{\Luk}\)-convergence coincides with uniform convergence.
\end{prop}
\begin{proof}
By Lemma~\ref{pointwise representation}, for the {\L}ukasiewicz t-norm,
\[
S_{\Luk}(u,u_i)=\inf_{x\in\mathbb{R}}\bigl(1-|u(x)-u_i(x)|\bigr)=1-\|u-u_i\|_\infty.
\]
Hence $S_{\Luk}(u,u_i)\to1$ if and only if $\|u-u_i\|_\infty\to0$; that is, $S_{\Luk}$-convergence coincides with uniform convergence. 
\end{proof}
\begin{remark}
  The function \(S\) is a \([0,1]\)-valued similarity rather than a metric,
since \(S(u,u)=1\). Its complement
\[
d_S(u,v):=1-S(u,v)
\]
is a bounded metric on \(\mathbb{E}^1\), with values in \([0,1]\), and the
maximum value \(1\) is attained. In the {\L}ukasiewicz case,
\[
d_{S_{\Luk}}(u,v)=\|u-v\|_\infty.
\]
\end{remark}
\begin{prop}\label{product norm in E^1}
   Let \(\&\) be the product t-norm. A net \(\{u_i\}_{i\in D}\) in \(\mathbb{E}^1\) \(S_{*}\)-converges to \(u\) if and only if
   \[
   \forall\,\varepsilon>0,\ \exists\,i_0\in D,\ \forall\,i\ge i_0,\ \forall\,x\in\mathbb{R}:\quad
   u_i(x)=u(x)=0\ \ \text{or}\ \ \frac{\min(u_i(x),u(x))}{\max(u_i(x),u(x))}>1-\varepsilon .
   \]
\end{prop}
\begin{proof}
By Lemma~\ref{pointwise representation}, \(S_{*}(u,u_i)=\inf_{x\in\mathbb{R}}(u_i(x)\leftrightarrow u(x))\), and for the product t-norm
\[
R(s,t)=
\begin{cases}
1, & s=t=0,\\[2pt]
\dfrac{\min(s,t)}{\max(s,t)}, & \max(s,t)>0.
\end{cases}
\]
Thus
\[
S_*(u,v)=\inf_{x\in\mathbb R}R(u(x),v(x)).
\]
Therefore \(S_{*}(u,u_i)\to1\) is equivalent to the displayed conditions.
\end{proof}

\begin{prop}\label{minimum t-norm in E^1}
   Let \(\&\) be the minimum t-norm. A net \(\{u_i\}_{i\in D}\) of \(\mathbb{E}^1\) \(S_{\wedge}\)-converges to \(u\) if and only if
   \[
   \forall\,r<1,\ \exists\,i_0\in D,\ \forall\,i\ge i_0,\ \forall\,x\in\mathbb{R}:\quad
   u_i(x)=u(x)\ \ \text{or}\ \ \min\bigl(u_i(x),u(x)\bigr)>r .
   \]
\end{prop}
\begin{proof}
By Lemma~\ref{pointwise representation} and the fact that for the minimum t-norm
\[
s\leftrightarrow t=\min(s,t)\ \ (s\ne t),\qquad s\leftrightarrow t=1\ \ (s=t),
\]
the assertion is immediate.
\end{proof}

\begin{prop}\label{convergence relation under different t-norm}
    Let  \(\{u_i\}_{i\in D}\) be a net in \(\mathbb{E}^1\).
    Then \[S_{\wedge}\text{-convergence}  \implies S_{*}\text{-convergence} \implies S_{\Luk}\text{-convergence}.\] 
\end{prop}
\begin{proof}
For any $a,b\in[0,1]$ one has
\[
a\leftrightarrow_{\wedge}b\ \le\ a\leftrightarrow_{*}b\ \le\ a\leftrightarrow_{\Luk}b .
\]
Indeed, if $a=b$ all three equal $1$; if $a\ne b$, writing $m=\min(a,b)$ and $M=\max(a,b)$,
\[
m\ \le\ \frac{m}{M}\ \le\ 1-(M-m),
\]
where the second inequality is equivalent to $m(1-M)\le M(1-M)$, i.e.\ $m\le M$. By Lemma~\ref{pointwise representation} and taking the infimum over $x\in\mathbb{R}$, we obtain
\[
S_{\wedge}(u,v)\ \le\ S_{*}(u,v)\ \le\ S_{\Luk}(u,v),
\]
whence $S_{\wedge}$-convergence $\Rightarrow$ $S_{*}$-convergence $\Rightarrow$ $S_{\Luk}$-convergence.
\end{proof}

\begin{thm}
Let \(\tau_{\Luk}\), \(\tau_{*}\), and \(\tau_{\wedge}\) denote the open ball topologies on \((\mathbb{E}^1,S)\) induced by the {\L}ukasiewicz, product, and G\"odel t-norms, respectively. Then
\[
\tau_{\Luk}\subseteq \tau_{*}\subseteq \tau_{\wedge}.
\]
\end{thm}
\begin{proof}
    By Proposition \ref{convergence relation under different t-norm}.
\end{proof}
\section{Completeness of the Fuzzy Order}
\begin{definition}
    Suppose $\{x_i\}_{i\in D}$ is a net and $a$ is an element of a \([0,1]\)-enriched category $X$. We say that
\begin{enumerate}[label={\textup{(\arabic*)}}]
    \item $\{x_i\}_{i\in D}$ is Cauchy (also called biCauchy) if
    \[
    \sup_{i\in D}\inf_{j,k\geq i} X(x_j,x_k) = 1.
    \]
    \item $a$ is a bilimit of $\{x_i\}_{i\in D}$ if for all $x\in X$,
    \[
    \sup_{i\in D}\inf_{j\geq i} X(x,x_j) = X(x,a), \quad \sup_{i\in D}\inf_{j\geq i} X(x_j,x) = X(a,x).
    \]
    \item  $\{x_i\}_{i \in D}$ is forward Cauchy if
		\[
		\sup_{i \in D} \inf_{k \geq j \geq i} X(x_j, x_k) = 1.
		\]
	\item  $b$ is a Yoneda limit of $\{x_i\}_{i \in D}$ if for all $y \in X$,
		\[
		X(b, y) = \sup_{i \in D} \inf_{j\ge i} X(x_j, y).
		\]
\end{enumerate}
It is clear that a net has at most one bilimit up to isomorphism.
\end{definition}
\begin{lem}[\cite{Zhang2024IntroductoryNO}, Lemma 8.11]\label{lem:bilimit}
    Let $\{x_i\}_{i\in D}$ be a Cauchy net and let $a$ be an element of a \([0,1]\)-enriched category $X$. The following are equivalent:
\begin{enumerate} [label={\textup{(\arabic*)}}]
    \item $a$ is a bilimit of $\{x_i\}_{i\in D}$.
    \item $\sup_{i\in D}\inf_{j\geq i} X(a,x_j) = 1$ and $\sup_{i\in D}\inf_{j\geq i} X(x_j,a) = 1$.
    \item $\{x_i\}_{i\in D}$ converges to $a$ in the open ball topology of the symmetrization of $X$.
\end{enumerate}
\end{lem}
\begin{lem}[\cite{Zhang2024IntroductoryNO}, Lemma 9.2]\label{lem:yoneda limit is bilimit}
    Suppose $\{x_i\}_{i \in D}$ is a Cauchy net and $b$ is an element of a \([0,1]\)-enriched category $X$. Then $b$ is a Yoneda limit of $\{x_i\}_{i \in D}$ if and only if it is a bilimit of $\{x_i\}_{i \in D}$.
\end{lem}

\begin{definition}
A \([0,1]\)-enriched category \(X\) is \emph{Cauchy complete} if every Cauchy net in \(X\) has a
bilimit. It is \emph{Smyth complete} if every forward Cauchy net in \(X\) has a bilimit, which is
then unique up to isomorphism and is the limit of the net in the open ball topology of the
symmetrization.
\end{definition}
\begin{prop}[\cite{Zhang2024IntroductoryNO}, Proposition 11.2]\label{forward Cauchy=Cauchy}
    If a \([0,1]\)-enriched category $X$ is Smyth complete, then it is separated and all of its forward Cauchy nets are Cauchy.
\end{prop}

\begin{thm}[\cite{Zhang2024IntroductoryNO}, Theorem 8.14]\label{thm:cauchy-complete-sequences}
A \([0,1]\)-enriched category \(X\) is Cauchy complete if and only if
each of its Cauchy sequences converges uniquely in the open ball topology of its
symmetrization.
\end{thm}
\begin{cor}[\cite{Zhang2024IntroductoryNO}, Corollary 11.6]\label{Cauchy complete and Smyth complete}
    For each \([0,1]\)-enriched category $X$, the following are equivalent:
\begin{enumerate}[label={\textup{(\arabic*)}}]
    \item $X$ is Cauchy complete.
    \item The symmetrization of $X$ is Smyth complete.
    \item The symmetrization of $X$ is Cauchy complete.
\end{enumerate}
\end{cor}

\begin{thm}[\cite{kouhui}, Theorem 4.3]\label{E^1_{[-M,M]} is Yoneda complete}
Let $\&$ be a  continuous t-norm. Then $(\mathbb{E}^1_{[-M,M]},P)$ is  Yoneda complete.    
\end{thm}
Kou and Xie also constructed an example showing that the full space \((\mathbb{E}^1,P)\) fails to be Yoneda complete under the {\L}ukasiewicz t-norm. Since Yoneda completeness implies Cauchy completeness~\cite{Zhang2024IntroductoryNO}, the failure of Yoneda completeness does not by itself preclude Cauchy completeness. The following example resolves  this question: it exhibits a Cauchy net in \((\mathbb{E}^1,P_{\Luk})\) with no bilimit, thereby proving that \(\mathbb{E}^1\) is not  Cauchy complete under the {\L}ukasiewicz t-norm.
\begin{exmp}\label{ex:cauchy-not-complete}
  Define 
  \[
        u_n(x)=\begin{cases}
            e^x, & -n\le x\le 0;\\
            0, & \text{otherwise},
        \end{cases}
 \]
    where \(n\in\mathbb{Z}^+\). Each $u_n\in\mathbb{E}^1$. For $m\ge n$ we have $u_m-u_n=e^x\chi_{[-m,-n)}$, so $\|u_m-u_n\|_\infty=e^{-n}$, and by Proposition~\ref{Lukasiewicz t-norm in E^1}, $S_{\Luk}(u_m,u_n)=1-e^{-n}$. For general $m,k$ one has $S_{\Luk}(u_m,u_k)=1-e^{-\min(m,k)}$, so $\inf_{m,k\ge n}S_{\Luk}(u_m,u_k)=1-e^{-n}\to1$. Hence
    \[
    \sup_n\inf_{m,k\ge n}S_{\Luk}(u_m,u_k)=1,
    \]
   so $\{u_n\}$ is Cauchy with respect to the symmetrization $S_{\Luk}$ (i.e.\ $\sup_n\inf_{m,k\ge n}S_{\Luk}(u_m,u_k)=1$), hence a Cauchy (biCauchy) net in $(\mathbb{E}^1,P_{\Luk})$.
 If $\{u_n\}$ had a bilimit $v\in\mathbb{E}^1$, then by Lemma~\ref{lem:bilimit} and Proposition~\ref{Lukasiewicz t-norm in E^1} we would have $\|u_n-v\|_\infty\to0$, so that $v$ must coincide with the pointwise limit $u_\infty(x)=e^x\chi_{(-\infty,0]}(x)$, whose support $(-\infty,0]$ is not compact; hence $u_\infty\notin\mathbb{E}^1$, a contradiction. Thus $\{u_n\}$ has no bilimit, and $\mathbb{E}^1$ is not Cauchy complete under the {\L}ukasiewicz t-norm.
\end{exmp}
\begin{lem}\label{lem:no-zero-divisors}
Let $\&$ be a continuous t-norm and let $a>0$. If $a$ lies in an Archimedean component of the ordinal-sum decomposition of $\&$ whose left endpoint is positive, or if $a$ is an idempotent with $a>0$, then
\[
a\mathbin{\&} b=0\ \Longrightarrow\ b=0 .
\]
\end{lem}
\begin{proof}
If $a$ and $b$ lie in the same component $[e_i,e_{i+1}]$ with $e_i>0$, then $a\mathbin{\&} b\ge e_i>0$. If they lie in different components, then $a\mathbin{\&} b=\min(a,b)>0$ whenever $b>0$. If $a$ is an idempotent with $a>0$, then for $b\le a$ one has $a\mathbin{\&} b=b$, while for $b>a$ one has $a\mathbin{\&} b\ge a>0$. In all cases $b>0$ implies $a\mathbin{\&} b>0$, so $a\mathbin{\&} b=0$ forces $b=0$.
\end{proof}

\begin{thm}\label{thm:cauchy-completeness}
     Let $\&$ be a continuous t-norm. \((\mathbb{E}^1, P) \) is Cauchy complete if and only if $\&$ is not   isomorphic to the {\L}ukasiewicz t-norm on \([0,1]\).
\end{thm}
\begin{proof}
\noindent\textbf{Necessity.}
Suppose that $(\mathbb{E}^{1},P)$ is Cauchy complete. We show that the continuous t-norm $\&$ cannot be isomorphic to the {\L}ukasiewicz t-norm.

Assume, on the contrary, that $\&$ is isomorphic to the {\L}ukasiewicz t-norm, i.e., there exists a strictly increasing bijection $\varphi:[0,1]\to[0,1]$ such that $\varphi(a\,\&\,b)=\max(\varphi(a)+\varphi(b)-1,0)$ for all $a,b\in[0,1]$. By Example \ref{ex:cauchy-not-complete}, $(\mathbb{E}^{1},P_{\Luk})$ is not Cauchy complete.
Now transport the construction via $\varphi$; set $v_n=\varphi^{-1}\circ u_n$, i.e.,
\[
v_{n}(x)=\begin{cases}
\varphi^{-1}(e^{x}), & -n\leq x\leq 0,\\[2pt]
0, & \text{otherwise}.
\end{cases}
\]
For any $j,k$,
\[
P(v_{j},v_{k})=\varphi^{-1}\big(P_{\Luk}(u_{j},u_{k})\big),
\]
since $\varphi$ is an order-isomorphism of $[0,1]$ commuting with arbitrary suprema and infima, and it carries the residuated implication of $\&$ to that of the {\L}ukasiewicz t-norm; concretely $v_n^\ell=\varphi^{-1}\circ u_n^\ell$, $v_n^r=\varphi^{-1}\circ u_n^r$, and $\varphi^{-1}(a)\to\varphi^{-1}(b)=\varphi^{-1}(a\to_{\Luk}b)$ for all $a,b\in[0,1]$.
Because $\varphi^{-1}(1)=1$,
\[
\sup_{i}\inf_{j,k\geq i}P(v_{j},v_{k})
=\varphi^{-1}\!\Big(\sup_{i}\inf_{j,k\geq i}P_{\Luk}(u_{j},u_{k})\Big)
=\varphi^{-1}(1)=1,
\]
so $\{v_{n}\}$ is a Cauchy net in $(\mathbb{E}^{1},P)$. If $\{v_{n}\}$ has a bilimit $v^{*}$ in $(\mathbb{E}^{1},P)$, then $\varphi\circ v^{*}$ would be a bilimit of $\{u_{n}\}$ in $(\mathbb{E}^{1},P_{\Luk})$, a contradiction. Hence $(\mathbb{E}^{1},P)$ is not Cauchy complete, contradicting the hypothesis. Therefore $\&$ is not isomorphic to the {\L}ukasiewicz t-norm.

\medskip
\noindent\textbf{Sufficiency.}
 Suppose \(\&\) is not isomorphic to the {\L}ukasiewicz t-norm. We prove that \((\mathbb{E}^1,P)\) is Cauchy complete. By Theorem~\ref{thm:cauchy-complete-sequences}, it suffices to show that every Cauchy sequence in \((\mathbb{E}^1,P)\) converges.

By the Mostert--Shields ordinal sum decomposition, $\&$ is an ordinal sum of continuous Archimedean t-norms, each isomorphic either to the {\L}ukasiewicz t-norm or to the product t-norm. We prove the following key estimate:
\begin{equation}\label{eq:key}
c:=\sup_{a>0}(a\to 0)<1,
\end{equation}
where $a\to0=a\leftrightarrow0$ since $0\to a=1$. Three cases arise.

\emph{Case 1.} There is no Archimedean component whose left endpoint is $0$. Then every $a>0$ lies in a component with positive left endpoint or is an idempotent, so Lemma~\ref{lem:no-zero-divisors} gives $a\to0=0$; hence $c=0$.

\emph{Case 2.} The bottom component $[0,e_1]$ is isomorphic to the product t-norm. For $a\in(0,e_1]$ the product t-norm has no zero divisors, so $a\to0=0$; for $a>e_1$, Lemma~\ref{lem:no-zero-divisors} again gives $a\to0=0$. Thus $c=0$.

\emph{Case 3.} The bottom component $[0,e_1]$ is isomorphic to the {\L}ukasiewicz t-norm. Let $\phi:[0,e_1]\to[0,1]$ be the isomorphism. For $0<a\le e_1$, we have $a\to0=\phi^{-1}(1-\phi(a))$, so $\sup_{a\in(0,e_1]}(a\to0)=e_1$; for $a>e_1$, Lemma~\ref{lem:no-zero-divisors} gives $a\to0=0$. Thus $c=e_1$. Since $\&$ is not isomorphic to the full {\L}ukasiewicz t-norm on $[0,1]$, we must have $e_1<1$, and therefore $c<1$.

Now let $\{u_{n}\}_{n\in \mathbb{N}}$ be a Cauchy sequence in $(\mathbb{E}^{1},P)$; that is,
\[
\sup_{n\in \mathbb{N}}\inf_{m,k\geq n}P(u_{m},u_{k})=1.
\]
Choose $\varepsilon>0$ with $1-\varepsilon>c$. There exists $n_{0}\in \mathbb{N}$ such that for all $m,k\geq n_{0}$,
\begin{equation}\label{eq:cauchy}
P(u_{m},u_{k})>1-\varepsilon.
\end{equation}
We claim that for all $k\ge n_0$, $[u_{n_0}]_0=[u_k]_0$.  If \([u]_0=[L_u,R_u]\), $u^\ell(x)>0$ for $x>L_u$ and $u^r(x)>0$ for $x<R_u$ (since $L_u=\inf\{x:u(x)>0\}$ and $R_u=\sup\{x:u(x)>0\}$). Set $[u_{n_0}]_0=[L_0,R_0]$ and $[u_k]_0=[L_k,R_k]$. If \(L_k<L_0\), choose \(x\in(L_k,L_0)\). Then
\(u_k^\ell(x)=a>0\) and \(u_{n_0}^\ell(x)=0\), so
\[
P(u_{n_0},u_k)
 \le a\to0
 \le c
 <1-\varepsilon,
\]
a contradiction.

The remaining cases \(L_k>L_0\), \(R_k>R_0\), and \(R_k<R_0\) are handled analogously: choosing \(x\) between the corresponding endpoints, we obtain the same contradiction. Hence $[u_{n_0}]_0=[u_k]_0$ for all $k\ge n_0$. Let this common compact support be $[L,R]$.

Set \(M>\max(|L|,|R|)\). Then every term \(u_n\), \(n\ge n_0\), belongs to
\(\mathbb{E}^{1}_{[-M,M]}\), and the tail
\(\{u_n\}_{n\ge n_0}\) is Cauchy in
\(\mathbb{E}^{1}_{[-M,M]}\). By Theorem~\ref{E^1_{[-M,M]} is Yoneda complete}, \(\mathbb{E}^{1}_{[-M,M]}\) is Yoneda complete.  Since the Cauchy tail \(\{u_n\}_{n\ge n_0}\) is in particular forward Cauchy, it admits a Yoneda limit, which is a bilimit by Lemma~\ref{lem:yoneda limit is bilimit}. Consequently, $(\mathbb{E}^1,P)$ is Cauchy complete.

\end{proof}
Theorem \ref{thm:cauchy-completeness} gives a complete characterization of Cauchy completeness of \((\mathbb{E}^1,P)\). We now turn to the stronger notion of Smyth completeness. By Proposition \ref{forward Cauchy=Cauchy}, every Smyth complete category has the property that all its forward Cauchy nets are Cauchy. The following example shows that under the {\L}ukasiewicz t-norm, this property does not hold for \((\mathbb{E}^1,P)\).
\begin{exmp}\label{ex: not smyth complete}
Define
\[
u_n(x) =
\begin{cases}
	1, & x \in \bigl[0,\,1-\tfrac1n\bigr], n\ge 1; \\
	0, & \text{otherwise}.
\end{cases}
\]
A direct computation gives \(u_n^\ell=\chi_{[0,\infty)}\) and \(u_n^r=\chi_{(-\infty,\,1-1/n]}\). Hence for \(j\le k\) one has \(u_j^\ell=u_k^\ell\) and \(u_j^r\le u_k^r\), so that \(u_k^\ell\to u_j^\ell=1\) and \(u_j^r\to u_k^r=1\) (using \(s\to t=1\iff s\le t\)); thus \(P(u_j,u_k)=1\).
Thus
\[
\sup_n\inf_{k\ge j\ge n}P(u_j,u_k)=1,
\]
i.e.\ \(\{u_n\}_{n\in\mathbb{N}}\) is a forward Cauchy net in \((\mathbb{E}^1,P_{\Luk})\).

However, it is not a Cauchy net: for \(i>j\) and \(x\in\bigl(1-\tfrac1j,\,1-\tfrac1i\bigr]\), we have \(u_i^r(x)=1\) and \(u_j^r(x)=0\), so \(P(u_i,u_j)\le u_i^r(x)\to u_j^r(x)=1\to0=0\), whence \(P(u_i,u_j)=0\) and \(\sup_n\inf_{i,j\ge n}P(u_i,u_j)=0\). By Proposition~\ref{forward Cauchy=Cauchy}, \((\mathbb{E}^1,P_{\Luk})\) is not Smyth complete.
\end{exmp}

Note that the calculation uses only the universal identities 
\(0 \to 1 = 1\) and \(1 \to 0 = 0\), which hold for the residuated implication of any t-norm. Hence the same construction works uniformly.
\begin{thm}
    For any continuous t-norm,  \((\mathbb{E}^1,P)\) is not Smyth complete.
\end{thm}
\begin{proof}
Let \(\{u_n\}_{n\in\mathbb{N}}\) be the net constructed in Example~\ref{ex: not smyth complete}. The argument given there uses only the universal identities \(0\to 1=1\) and \(1\to 0=0\), which hold for the residuated implication of any continuous t-norm. Hence, verbatim, \(\{u_n\}_{n\in\mathbb{N}}\) is a forward Cauchy net but not a Cauchy net in \((\mathbb{E}^1,P)\). By Proposition~\ref{forward Cauchy=Cauchy}, \((\mathbb{E}^1,P)\) is not Smyth complete.
\end{proof}

\begin{cor}
       Let $\&$ be a continuous t-norm. \((\mathbb{E}^1,S) \) is Smyth complete if and only if $\&$ is not isomorphic to the {\L}ukasiewicz t-norm on \([0,1]\).

\end{cor}
\begin{proof}
By Theorem \ref{thm:cauchy-completeness} and Corollary \ref{Cauchy complete and Smyth complete}.    
\end{proof}

\section{Conclusion}
In this paper, we have carried out a systematic investigation into the completeness properties of the fuzzy order \(P(u,v)\) on the fuzzy number space \(\mathbb{E}^1\), within the framework of \([0,1]\)-enriched categories. Introduced recently by Kou and Xie, this fuzzy order employs the residuated implication of a continuous t-norm to compare fuzzy numbers in a logically grounded manner, thereby unifying and generalizing the natural orders on both the real line and the space of interval numbers.

First, we studied \(S\)-convergence under the open ball topology induced by the symmetrized fuzzy order \(S\), and obtained a complete characterization for three fundamental continuous t-norms. Under the {\L}ukasiewicz t-norm, \(S\)-convergence coincides exactly with uniform convergence in the supremum metric \(d_\infty(u,v)=\|u-v\|_\infty\); under the product and G\"odel (minimum) t-norms, equivalent pointwise characterizations are also obtained. These results establish a chain of inclusions among the three induced open ball topologies:
\[
\tau_{\Luk}\ \subseteq\ \tau_{*}\ \subseteq\ \tau_{\wedge}.
\]

Second, a sharp dichotomy concerning completeness is established. Regardless of the choice of continuous t-norm, \((\mathbb{E}^1,P)\) fails to be Smyth complete; moreover, \((\mathbb{E}^1,P)\) is Cauchy complete if and only if the continuous t-norm is not isomorphic to the {\L}ukasiewicz t-norm. Consequently, the symmetrized fuzzy order \((\mathbb{E}^1,S)\) is Smyth complete if and only if the continuous t-norm is not isomorphic to the {\L}ukasiewicz t-norm. In other words, although the {\L}ukasiewicz t-norm is the most natural choice from a metric perspective, it is the unique source of incompleteness: under any other continuous t-norm, \((\mathbb{E}^1,P)\) is Cauchy complete and \((\mathbb{E}^1,S)\) is Smyth complete.

Together with the Yoneda completeness results of Kou and Xie for the uniformly bounded fuzzy number space \((\mathbb{E}^1_{[-M,M]},P)\), these findings complete the picture of completeness properties for fuzzy orders on fuzzy number spaces, and settle the open problem left in the literature concerning the completeness of the full space \((\mathbb{E}^1,P)\).



\bibliographystyle{abbrv}

\end{document}